\documentclass{amsart}

\usepackage[utf8]{inputenc}
\usepackage{amssymb}
\usepackage{amsfonts}
\usepackage{amsmath}
\usepackage{amsthm}
\usepackage{enumerate}
\usepackage{mathrsfs}
\usepackage{comment}
\usepackage{mathtools}
\usepackage{tikz-cd}
\usepackage{bbm}
\usepackage{hyperref}
\usepackage{cleveref}

\newcommand{\Z}{\mathbb Z}
\newcommand{\Q}{\mathbb Q}
\newcommand{\R}{\mathbb R}
\newcommand{\C}{\mathbb C}

\let\conjugate\overline
\let\clos\overline

\let\epsilon\varepsilon
\let\phib\phi
\let\phi\varphi

\newcommand*{\mb}[1]{\ensuremath{\mathbb{#1}}}
\newcommand*{\mc}[1]{\ensuremath{\mathcal{#1}}}
\newcommand*{\mf}[1]{\ensuremath{\mathfrak{#1}}}

\newcommand*{\bs}{\backslash}

\DeclareMathOperator{\Hom}{Hom}
\DeclareMathOperator{\Spec}{Spec}
\DeclareMathOperator{\Lie}{Lie}
\DeclareMathOperator{\End}{End}
\DeclareMathOperator{\Tr}{Tr}
\DeclareMathOperator{\Res}{Res}
\DeclareMathOperator{\Fil}{Fil}
\DeclareMathOperator{\dR}{dR}

\DeclareMathOperator{\Gal}{Gal}

\DeclareMathOperator{\Ind}{Ind}

\DeclarePairedDelimiter{\norm}{\lVert}{\rVert}
\DeclarePairedDelimiter{\abs}{\lvert}{\rvert}
\DeclarePairedDelimiter{\inner}{\langle}{\rangle}
\DeclarePairedDelimiter{\set}{\{}{\}}
\DeclarePairedDelimiter{\paren}{(}{)}

\title{Towards the Colmez Conjecture}
\author{Roy Zhao}
\address{Yau Mathematics and Sciences Center, Tsinghua University, Beijing, China.}
\email{rhzhao@tsinghua.edu.cn}


\newtheorem{theo}{Theorem}[section]
\newtheorem{coro}[theo]{Corollary}
\newtheorem{conj}[theo]{Conjecture}
\newtheorem{prop}[theo]{Proposition}
\newtheorem{lemm}[theo]{Lemma}

\theoremstyle{definition}
\newtheorem{defi}[theo]{Definition}
\newtheorem{exam}[theo]{Example}

\theoremstyle{remark}
\newtheorem*{rema}{Remark}

\begin{document}
\begin{abstract}
	We prove a collection of results involving Colmez's periods and the Colmez Conjecture.
	Using Colmez's theory of periods of CM abelian varieties, we propose a definition for the height of a partial CM-type and prove that the Colmez conjecture follows from an arithmetic period formula for surfaces.
	We give an explicit conjecture for the form of this period formula, which relates the height of special points on a Shimura surface with special values of $L$-functions.
	Further, we relate the heights of periods given by Colmez to arithmetic degree of Hermitian line bundles and thus give a formulation of Colmez's full conjecture in geometric terms.
\end{abstract}

\maketitle

\section{Introduction}

In his 1993 article \cite{Col93}, Colmez explicitly constructs a height function for certain periods of a CM abelian variety and conjectures that this height can be expressed in terms of special values of logarithmic derivatives of $L$-functions.
One consequence of this conjecture, and how this conjecture is usually presented in the literature, is an explicit formula between the Faltings height of a CM abelian variety and the logarithmic derivatives of $L$-functions.
This conjecture is equivalent to the Chowla--Selberg formula in the case of elliptic curves.
We first precisely state the conjecture following \cite{Col93}.

Let $\Q^{CM}$ be the maximal CM-extension of $\Q$ and let $c$ denote complex conjugation on $\Q^{CM}$.
Let $\mc{CM}$ be the set of locally constant functions $f \colon \Gal(\clos{\Q}/\Q) \to \Q$ that factor through $\Gal(\Q^{CM}/\Q)$ and satisfy the condition that $f(x) + f(cx)$ is independent of $x \in \Gal(\clos{\Q}/\Q)$.
Define $\mc{CM}^0 \subset \mc{CM}$ as the subset of class functions, namely those functions $f$ such that $f(xyx^{-1}) = f(y)$ for all $x, y \in \Gal(\Q^{CM}/\Q)$.

Tensoring up to $\C$, the set of class functions $\mc{CM}^0 \otimes \C$ admits a basis in terms of Artin characters.
Let $\mu_{Art}$ be the $\C$-linear functional on $\mc{CM}^0$ given by its value on an Artin character $\chi$ as $\mu_{Art}(\chi) = \log f_\chi$, where $f_\chi$ is the conductor of $\chi$.
For $s \in \C$, let $Z(\cdot, s)$ be another $\C$-linear functional on $\mc{CM}^0$ given on Artin characters as $Z(\chi, s) = L'(\chi, s)/L(\chi, s)$.

Colmez defines another $\C$-linear functional on $\mc{CM}^0$ defined by using another basis in terms of CM-types of CM-fields.
Fix an embedding of $\clos{\Q} \subset \C$.
If $E/\Q$ is a number field, define $H_E \coloneqq \Hom(E, \clos{\Q}) = \Hom(E, \C)$.
For a CM-field $E$, CM-type $\Phi \subset H_E$, and embedding $\tau \in H_E$, define the function $a_{\Phi, \tau} \colon \Gal(\clos{\Q}/\Q) \to \C$ by the formula
\[a_{\Phi, \tau}(x) = \begin{cases}
    1 & \text{if $x\tau \in \Phi$,}\\
    0 & \text{otherwise.}
\end{cases}\]
Let $a_{\Phi, \tau}^0 \in \mc{CM}^0$ be its average under conjugation, namely
\[a_{\Phi, \tau}^0 = \frac{1}{[K : \Q]} \sum_{\sigma \in H_K} a_{\sigma \Phi, \sigma \tau}\]
for a number field $K$ containing all conjugates of $E$.
Finally set $A_{\Phi} = \sum_{\tau \in \Phi} a_{\Phi, \tau}$ and $A_{\Phi}^0 = \sum_{\tau \in \Phi} a_{\Phi, \tau}^0$.

Let $A$ be a CM abelian variety of type $(\mc{O}_E, \Phi)$ and let $\tau \in \Phi$.
Colmez uses the pairing between $H_{\dR}^1(A/\C)$ and $H_1(A(\C), \Z)$ to define a $\tau$-component of the Faltings height, which we denote $ht(\Phi, \tau)$.
He then proves that $ht(a_{\Phi, \tau}^0) \coloneqq ht(\Phi, \tau)$ can be extended to a linear height function, denoted $ht$, on $\mc{CM}^0$.
Moreover, it can be related to the Faltings height of a CM abelian variety.
The precise definition of the Faltings height of an abelian variety is given in Definition \ref{defn:Faltings Height}.

\begin{theo}[{{\cite[Conj. II.2.10]{Col93}}}]
	If $A$ is a CM abelian variety of type $(\mc{O}_E, \Phi)$, then
	\[h_{\mathrm{Falt}}(A) = ht(A_{\Phi}^0) - \frac{1}{2}\mu_{Art}(A_{\Phi}^0).\]
	Since the right hand side depends only on the CM-type $\Phi$, we write $h(\Phi)$ to denote the Faltings height of any (hence all) abelian varieties with CM of type $(\mc{O}_E, \Phi)$.
\end{theo}

Colmez conjectures that the height functions given by $ht$ and $Z(\cdot, s)$ are the same.
\begin{conj}[{{\cite[Conj. II.2.11]{Col93}}}]\label{conj:Original Colmez Conjecture}
    If $a \in \mc{CM}^0$ is a class function, let its dual $a^\vee \colon \Gal(\clos{\Q}/\Q) \to \C$ be the function given by $a^\vee(g) = a(g^{-1})$ for all $g \in \Gal(\clos{\Q}/\Q)$.
    Then
    \[ht(a) = -Z(a^\vee, 0)\]
    for all $a \in \mc{CM}^0$.
\end{conj}

However the version commonly stated in the literature is in terms of Faltings heights and $A_\Phi$.
Note that $A_\Phi^\vee = A_\Phi$ and proving the Colmez Conjecture for all CM abelian varieties will only prove it for $a \in \mc{CM}^0$ such that $a = a^\vee$.

\begin{conj}[Colmez Conjecture]\label{conj:Colmez}
    \[h(\Phi) = -Z(A_\Phi^0, 0) - \frac{1}{2}\mu_{Art}(A_\Phi^0).\]
\end{conj}

In the same work, Colmez proved Conjecture \ref{conj:Colmez} for all abelian CM fields up to a rational multiple of $\log 2$, which was later fully proven by Obus (\cite{Obu13}).
Yang proved the conjecture when $A$ is an abelian surface, and thus when $\abs{\Phi} = 2$ (\cite{Yan10}).
Colmez also gave the form of the conjecture when both sides are averaged over all CM-types $\Phi$ of a fixed CM field $E$.
Stated in that form, it says that
\[\frac{1}{2^g}\sum_\Phi h(\Phi) = -\frac{1}{2} \frac{L'(\chi_{E/F}, 0)}{L(\chi_{E/F}, 0)} - \frac{1}{8} \log(d_{E/F}d_E) + \frac{[E:\Q]}{4}\log 2\pi,\]
where $F$ is the maximal totally real subfield of $E$ and $\chi_{E/F}$ is the character associated with the quadratic extension.
This averaged version was proven independently by Yuan and Zhang (\cite{Yua18}) and a group of Andreatta, Goren, Howard, and Madapusi (\cite{And18}).
This average result was extended to certain ``unitary CM-types of signature $(n - 1, 1)$'' by Yang and Yin (\cite{Yang18}).
Moreover, using this averaged result, Barquero-Sanchez, Masri, and Thorne were able to prove that $100\%$ of CM abelian varieties satisfy the Conjecture \ref{conj:Colmez} (\cite{BMT23}).
This averaged version also has other far-ranging consequences.
Tsimerman showed that the averaged Colmez Conjecture implies the Andr\'e--Oort Conjecture for the moduli space of abelian varieties (\cite{Tsi18}) and later with Pila, Shankar, Esnault, and Groechenig proved that it implied the full Andr\'e--Oort Conjecture for all Shimura varieties (\cite{Pil24}).

In this article, we prove two main results which serve as steps towards proving Conjecture \ref{conj:Colmez} as well as Colmez's original conjecture that $ht(a) = Z(a^\vee, s)$ for all $a \in \mc{CM}^0$.
We define the height of a subset $\phib \subset \Phi$ of a CM-type, which we call a partial CM-type, and express the Faltings height of a CM abelian variety in terms of heights of partial CM-types of order $2$ in Proposition \ref{prop:Full in terms of Partial}.
A consequence of this result is that Conjecture \ref{conj:Colmez} can be reduced to proving explicit identities for heights of partial CM-types of cardinality $2$ (see Theorem \ref{thm:Surface implies Colmez}).
Combining Theorem \ref{thm:Surface implies Colmez} with \cite{Zha23} gives Corollary \ref{coro:Surface implies Colmez}, which reduces Conjecture \ref{conj:Colmez} to an explicit height formula for CM points on a quaternionic Shimura surface.
Moreover, we make some progress towards giving a geometric reformulation for the period formula given by Conjecture \ref{conj:Original Colmez Conjecture}.
In Theorem \ref{theo:Intro Comparison Theorem} (also Theorem \ref{thm:ColmezYuanComparison}), we express the periods given in \cite{Col93} in terms of Arakelov geometry, and by doing so, are able to recover the observation of \cite{Yua18} that the height of nearby CM-types depends only on the CM field, and not on the nearby pair (Corollary \ref{cor:Nearby Height Ind of Pair}).

We now give a more precise description of our results.
For the first result, let $\phib \subset \Phi$ be a partial CM-type.
We use the height function $ht$ to propose an explicit definition for the height of a partial CM-type $\phib$.
Note that we can simplify $A_\Phi$ for a CM-type $\Phi$ as
\[A_\Phi(g) = \abs{\Phi \cap g\Phi}.\]
However, the function $g \mapsto \abs{\phib \cap g \phib}$ is not in $\mc{CM}^0$ because $\abs{\phib \cap g \phib} + \abs{\phib \cap cg\phib}$ is not independent of $g$.
We fix this to define the height of a partial CM-type.

\begin{defi}\label{defi:Definition of height of partial CM type}
    Let $\phib \subset H_E$ be a partial CM-type of $E$ and let $\Sigma \subset H_F$ denote the set of places obtained by restricting those in $\phib$ to $F$.
    Let $A_\phib \in \mc{CM}$ be given by the formula
    \[A_\phib(x) = \abs{\phib \cap x\phib} + \frac{1}{2}\abs{\Sigma \cap x \Sigma^c}.\]
    The height of the partial CM-type $\phib$ is
    \[h(\phib) \coloneqq ht(A_\phib^0) - \frac{1}{2}\mu_{Art}(A_\phib^0).\]
\end{defi}

We give formulas expressing the Faltings height of a CM-abelian variety in terms of heights of partial CM-types of size $2$ (see Proposition \ref{prop:Partial in terms of Full}).
This means that Conjecture \ref{conj:Colmez} can be reduced to proving it for partial CM-types of size $2$, and we precisely determine the exact form of the $L$-functions and their conductors.

\begin{theo}\label{thm:Surface implies Colmez}
    Suppose that $E/\Q$ is a Galois CM-extension and if $\sigma, \tau \in H_E$ are distinct places so that $\sigma \neq \conjugate{\tau}$, let $x_{\sigma, \tau} \in \Gal(E/\Q)$ be the element such that $x_{\sigma, \tau}\sigma = \tau$.
    Let $c \in \Gal(E/\Q)$ denote complex conjugation.

    Conjecture \ref{conj:Colmez} holds for every CM-type of every CM-subfield of $E$ if and only if for every choice of $\sigma, \tau \in H_E$ such that $\sigma \neq \conjugate{\tau}$, we have
    \begin{align*}
        h(\phib) = & -  \frac{1}{2}\sum_{\chi}\paren*{\chi(x_{\sigma, \tau}) + \chi(x_{\sigma, \tau}c) + \chi(x_{\sigma, \tau})^{-1} + \chi(x_{\sigma, \tau}c)^{-1}}\frac{L'(\chi, 0)}{L(\chi, 0)}\\
        &-\frac{1}{g}\frac{L'(\chi_{E/F}, 0)}{L(\chi_{E/F}, 0)} - \frac{1}{4[E:\Q]}\log d_E + \frac{1}{4}\log d_\phib - \frac{1}{8} \log d_\Sigma + \log 2\pi,
    \end{align*}
    where $\chi$ runs over all irreducible characters of $\Gal(E/\Q)$, the partial CM-type is $\phib = \set{\sigma, \tau}$ with $\Sigma = \set{\sigma|_F, \tau|_F} \subset H_F$, and $d_\phib$ and $d_\Sigma$ are certain root discriminants defined in Definition \ref{def:CM Type Determinant}.
\end{theo}
\begin{proof}
    This immediately follows from Proposition \ref{prop:Full in terms of Partial}, Corollary \ref{coro:Suffice to prove for pairs}, and Proposition \ref{prop:Height of Partial CM Type L Func}.
\end{proof}

Using previous work of the author \cite{Zha23}, we can restate the formula given in the previous theorem as a period formula for quaternionic Shimura surfaces.
We now precisely state this period formula, and refer to \cite{Del71} for more details about Shimura varieties.

Let $E, F, \phib, \Sigma$ be as in the setting of the previous theorem.
Let $B/F$ be a quaternion algebra which permits an embedding $E \to B$ and assume that at the archimedean places, $B$ is ramified exactly at the infinite places corresponding to $\Sigma$.
Let $G/\Q$ be the algebraic group given by $G \coloneqq \Res_{F/\Q} B^\times$.
Setting $\mc{H}^{\pm}$ to be the upper and lower half-planes, the Shimura datum $(G, (\mc{H}^{\pm})^\Sigma)$ gives a tower of Shimura varieties
\[X_U(\C) = G(\Q) \bs (\mc{H}^\pm)^\Sigma \times G(\mb{A}_f) / U\]
for every compact open subgroup $U \subset G(\mb{A}_f)$.

The complex variety $X_U$ has an algebraic model defined over a subfield of $E$.
It can be defined over the field $F_\Sigma$, which is the fixed field of all the elements of $\Gal(E/\Q)$ which fix $\Sigma \subset H_F$.
To simplify notation, we base change up to $E$ and view $X_U$ as a variety defined over $\Spec E$.
In \cite{Zha23}, for $U = \prod_p U_p \subset G(\mb{A}_f)$ maximal, an integral model of $X_U$ is constructed over $\Spec \mc{O}_E$, which we denote $\mc{X}_U$.

Let $\widehat{\mc{L}_U} \coloneqq (\mc{L}_U, \norm{\cdot})$ be the arithmetic Hodge bundle on $\mc{X}_U$, which consists of the pair of Hodge bundle $\mc{L}_U$ over $\Spec \mc{O}_E$ and Hermitian metric on $\mc{L}_{U, \C} \cong \Omega_{(\mc{H}^{\pm})^\Sigma}^\Sigma$ given by $\norm{\bigwedge_{\sigma \in \Sigma} dz_\sigma} = \prod_{\sigma \in \Sigma} 2\Im(z_\sigma)$.

The embedding $E \to B$ gives a map $\Res_{E/\Q}\mb{G}_m \to G$ and the choice of partial CM-type $\phib$ gives us the image of cocharacter of $\Res_{E/\Q} \mb{G}_m$ in $(\mc{H}^{\pm})^\Sigma$ by specifying the upper or lower half-plane, and thus the combination gives a mapping of Shimura datum and a CM-point $P_U \in X_U(\clos{\Q})$.
An integral model $\mc{X}_U$ of $X_U$ was constructed in \cite{Zha23} and let $\mc{P}_U$ denote the Zariski closure of $P_U$ in $\mc{X}_U$.
Let $h_{\widehat{\mc{L}_U}}(\mc{P}_U)$ be the Arakelov degree.
Then a consequence of \cite[Thm. 1.1]{Zha23} (and the proof of \cite[Thm. 7.5]{Zha23} for the specific log terms) is the equivalence of the following arithmetic period formula with the Colmez Conjecture.

\begin{coro}\label{coro:Surface implies Colmez}
    Suppose that $U = \prod_p U_p$ is a maximal compact subgroup of $G(\mb{A}_f)$. Then Conjecture \ref{conj:Colmez} holds for all CM-types of $E$ and all of its CM-subfields if and only if for all choices of $\sigma, \tau$ we have
    \begin{align*}
        \frac{1}{2}h_{\widehat{\mc{L}_U}}(\mc{P}_U) = & -  \frac{1}{2}\sum_{\chi}\paren*{\chi(x_{\sigma, \tau}) + \chi(x_{\sigma, \tau}c) + \frac{1}{\chi(x_{\sigma, \tau})} + \frac{1}{\chi(x_{\sigma, \tau}c)}}\frac{L'(\chi, 0)}{L(\chi, 0)}\\
        &-\frac{1}{g}\frac{L'(\chi_{E/F}, 0)}{L(\chi_{E/F}, 0)} - \frac{1}{8g}\log d_Ed_B^{-2} + \frac{1}{8}\log d_\phib^2d_\Sigma + \log 2\pi,
    \end{align*}
\end{coro}

Then towards Colmez original conjecture, we express his height of CM periods in geometrical terms as the arithmetic degree of a certain line bundle.
Namely in \cite{Yua18}, for a CM-type $(E, \Phi)$ and $\tau \in \Phi$, they define the $\tau$-component of the height $h(\Phi)$, which we denote by $h(\Phi, \tau)$, to be the Arakelov degree of a metrized line bundle.
We compare this to Colmez's height function $ht(\Phi, \tau)$ which was defined in terms of period pairings.
We prove the following theorem, which appears as Theorem \ref{thm:ColmezYuanComparison} later in the article.
\begin{theo}\label{theo:Intro Comparison Theorem}
    Let $E/\Q$ be a CM field with absolute discriminant $d_E \in \Z$.
    Let $\Phi$ be a CM-type of $E$ and $\tau \in \Phi$.
    Let $h(\Phi, \tau)$ be as defined in \cite[Thm. 2.2]{Yua18} and $ht(\Phi, \tau)$ be as defined in \cite[Lem. II.2.9]{Col93}.
    Then
    \[h(\Phi, \tau) = ht(\Phi, \tau) + \frac{1}{2} \log 2\pi + \frac{1}{2[E:\Q]}\log \abs{d_E} - \mu_{Art}(a_{\Phi, \tau}^0).\]
\end{theo}
We hope that this gives an avenue for proving Colmez's original conjecture for all functions in $\mc{CM}^0$.

\subsection{Structure of the Article}

In Section \ref{sec:Partial CM Type and Colmez}, we prove many results about our height of a partial CM-type.
Namely, we express heights of partial CM-types in terms of heights of (full) CM-types in Proposition \ref{prop:Full in terms of Partial} and heights of CM-types in terms of partial CM-types \ref{prop:Partial in terms of Full}.
Then, we explicitly compute the height of a partial CM-type of size $2$ assuming the Colmez Conjecture in Proposition \ref{prop:Height of Partial CM Type L Func}.
Finally, we give an example showing that these heights of partial CM-types are not comparable in Example \ref{exam:Height of pair large}.
In Section \ref{sec:Decomposition of Faltings Height}, we recall the definition of the Faltings height of an abelian variety.
Moreover, for a CM-type $\Phi$ and $\tau \in \Phi$, we recall the geometric definition of the $\tau$-component of $h(\Phi)$ given in \cite[Thm. 2.2]{Yua18}, and the definition in terms of periods given in \cite[Lem. II.2.9]{Col93}.
Then, in Section \ref{sec:Comparison of YZ and Colmez}, we give an explicit comparison identity between the geometric and period definitions (Theorem \ref{thm:ColmezYuanComparison}).

\subsection{Acknowledgements}
We wish to thank Pierre Colmez for encouraging me to propose a definition for a partial CM-type and Shou-Wu Zhang for suggesting that the Colmez Conjecture can be reduced to a $2$-dimensional period formula.
We would also like to thank Ziqi Guo for asking about how heights of partial CM-types vary for a fixed CM-field $E$.
Finally, we would like to thank the anonymous referee for their careful reading and suggestions, improving the presentation of this article.

\section{Partial CM Types and the Colmez Conjecture}\label{sec:Partial CM Type and Colmez}

Let $E$ be a CM-field, which is a totally imaginary quadratic extension of a totally real number field, and let its degree be $[E : \Q] = 2g$ and its ring of integers be $\mc{O}_E$.
There exists a unique complex conjugation automorphism of $E$, which we denote $c$, and for $\sigma \in H_E$, let $\conjugate{\sigma} \coloneqq \sigma \circ c$.
A CM-type $\Phi \subset H_E$ is a subset such that $\Phi \sqcup \conjugate{\Phi} = H_E$.
Note that giving a CM-type $\Phi$ is the same as giving an isomorphism $E \otimes \R \to \C^g$.

\begin{defi}
    An abelian variety $A$ has complex multiplication of type $(\mc{O}_E, \Phi)$ if there exists an embedding $\iota \colon \mc{O}_E \to \End(A)$ such that $\Lie(A) \cong E \otimes \R$ as $\mc{O}_E$-modules.
\end{defi}

Let $A$ be an abelian variety with CM of type $(\mc{O}_E, \Phi)$ and let $\tau \in \Phi$.
Colmez uses the pairing between $H_{\dR}^1(A/\C)$ and $H_1(A(\C), \Z)$ to define a $\tau$-component of the Faltings height, which we denote $ht(\Phi, \tau)$.
He then proves that $ht(\Phi, \tau)$ can be extended to a linear height function on $\mc{CM}^0$.

\begin{theo}[{\cite[Thm. II.2.10]{Col93}}]\label{thm:ColmezHeightFunction}
    \begin{enumerate}[(i)]
        \item There exists a unique $\Q$-linear map, denoted $ht$, from $\mc{CM}^0$ to $\R$ such that, if $E$ is a CM-field and $\tau \in H_E$ and $\Phi$ a CM-type of $E$, then $ht(a_{\Phi, \tau}^0) = ht(\Phi, \tau)$.
        
        \item If $A$ is a CM abelian variety of type $(\mc{O}_E, \Phi)$, the Faltings height of $A$ depends only on $(\mc{O}_E, \Phi)$ and is given by the formula
        \[h(\Phi) \coloneqq h(A) = ht(A_{\Phi}^0) - \frac{1}{2}\mu_{Art}(A_{\Phi}^0).\]
    \end{enumerate}
\end{theo}

We also define an auxiliary function $b_{\tau, \rho} \in \mc{CM}^0$ for $\tau, \rho \in H_E$ that will be useful later as
\[b_{\tau, \rho}(x) = \begin{cases}
    \frac{1}{2} & \text{if $x\tau = \rho$,}\\
    -\frac{1}{2} & \text{if $x\tau = \conjugate{\rho}$,}\\
    0 & \text{otherwise.}
\end{cases}\]

We note that we can write this height explicitly in terms of Faltings heights of CM abelian varieties.

\begin{prop}\label{prop:Partial in terms of Full}
    Let $h(\Phi)$ denote the Faltings height of an abelian variety with CM of type $(\mc{O}_E, \Phi)$.
    Let $[E : \Q] = 2g$.
    Then
    \[h(\phib) = \frac{1}{2^{g - \abs{\phib}}}\sum_{\Phi \supset \phib} h(\Phi) - \frac{g - \abs{\phib}}{g2^g} \sum_{\Phi'} h(\Phi'),\]
    where the first sum is taken over all CM-types of $E$ containing $\phib$ and the second is an unrestricted sum of CM-types.
\end{prop}
\begin{proof}
    Translating everything to class functions using Theorem \ref{thm:ColmezHeightFunction}, it suffices to show that
    \[A_\phib^0 = \frac{1}{2^{g - \abs{\phib}}}\sum_{\Phi \supset \phib} A_\Phi^0 - \frac{g - \abs{\phib}}{g2^g}\sum_{\Phi'} A_{\Phi'}^0.\]
    The left hand side is
    \[A_\phib = \frac{\abs{\phib}}{2} + \sum_{\tau, \rho \in \phib} b_{\tau, \rho} + \sum_{\tau|_F \in \Sigma} \sum_{\rho|_F \not \in \Sigma} \frac{1}{2}b_{\tau, \rho} = \frac{\abs{\phib}}{2} + \sum_{\tau, \rho \in \phib} b_{\tau, \rho}.\]
    The second summation is $0$ because $b_{\tau, \rho} + b_{\tau, \conjugate{\rho}} = 0$ for any $\tau, \rho \in H_E$.
    Meanwhile, noting that $A_\Phi = g/2 + \sum_{\tau, \rho \in \Phi} b_{\tau, \rho}$, we can simplify the right hand side as
    \begin{align*}
    	\frac{1}{2^{g - \abs{\phib}}}\sum_{\Phi \supset \phib} A_\Phi - \frac{g - \abs{\phib}}{g2^g}\sum_{\Phi'} A_{\Phi'} = &\frac{\abs{\phib}}{2} + \sum_{\tau, \rho \in \phib} b_{\tau, \rho} \\
    	&+ \frac{1}{2} \sum_{\tau|_F \not \in \Sigma} b_{\tau, \tau} - \frac{g - \abs{\phib}}{2g} \sum_{\tau \in H_E} b_{\tau, \tau}.
    \end{align*}
    It is clear that $b_{\tau, \tau}$ is independent of $\tau$ (and moreover is equal to $\frac{1}{2g}\Ind_{E/F}^{E/\Q} \chi_{E/F}$), and thus the last two summations cancel giving the desired equality.
\end{proof}

Moreover, using the comparison between the Faltings height and Colmez's height function, we can express the Faltings height of a CM abelian variety in terms of partial CM-types of order $2$.
Note that this cannot be directly deduced from the proposition above.

\begin{prop}\label{prop:Full in terms of Partial}
    Let $\Phi$ be a CM-type of $E$.
    Then
    \[h(\Phi) = \sum_{(\tau, \rho)} h(\set{\tau, \rho}) - \frac{g(g - 1)}{2} \log 2\pi,\]
    where the first sum is taken over all unordered pairs of $\tau, \rho \in \Phi$ (not necessarily distinct).
\end{prop}
\begin{proof}
    Again, it suffices to show that the same holds in terms of functions in $\mc{CM}^0$.
    The left hand side is
    \[A_\Phi = \frac{g}{2} + \sum_{\rho_1, \rho_2 \in \Phi} b_{\rho_1, \rho_2}.\]
    On the other hand if $\tau \neq \rho$, then
    \[A_{\set{\tau, \rho}} = 1 + b_{\tau, \rho} + b_{\rho, \tau}.\]
    Otherwise $A_{\set{\tau}} = \frac{1}{2} + b_{\tau, \tau}$.
    Thus, we have
    \[A_\Phi - \sum_{(\tau, \rho)} A_{\set{\tau, \rho}} = - \frac{g(g - 1)}{2}.\]
    Finally, by looking at the definition of $ht(\Phi, \tau)$, we see that
    \[ht(1) = ht(\Phi, \tau) + ht(\Phi, \conjugate{\tau}) = \log 2\pi.\]
\end{proof}

Using this language and the fact that $\sum_{\Phi} A_\Phi = \frac{g}{2} + gb_{\tau, \tau}$ for some (hence any) $\tau \in H_E$, we can state the averaged Colmez Conjecture as follows.

\begin{theo}[{{\cite[Thm A]{And18}, \cite[Thm 1.1]{Yua18}}}]\label{thm:Averaged Colmez}
    \[ht(b_{\tau, \tau}) = -Z(b_{\tau, \tau}).\]
\end{theo}

\begin{coro}\label{coro:Suffice to prove for pairs}
    To prove Conjecture \ref{conj:Colmez}, it suffices to prove the result for all pairs of distinct embeddings $\rho, \tau \in H_E$.
    Namely, to prove that
    \[ht(A_{\set{\rho, \tau}}^0) = -Z(A_{\set{\rho, \tau}}^0, 0).\]
\end{coro}

We now precisely compute the conjectured value of the height of a partial CM-type $h(\set{\rho, \tau})$ by computing log discriminant term for the height of a partial CM-type.

\begin{defi}\label{def:CM Type Determinant}
    Let $S \subset \Hom(E, \C)$ be a non-empty subset.
    Let $L \subset \C$ be a field containing all conjugates of $E$.
    We can decompose $E \otimes_\Q L \cong \prod_{\sigma \colon E \to L} L_\sigma$ and let
    \[f_S \colon E \otimes_\Q L \cong \prod_{\sigma \colon E \to L} L_\sigma \to \prod_{\sigma \in S} L_\sigma\]
    be the composition of the isomorphism with the projection onto the $S$-coordinates.
    Let $R_S$ denote the image of $\mc{O}_E \otimes_\Z \mc{O}_L$ under $f_S$ and let $\mf{d}_S$ be the relative discriminant over $\mc{O}_L$.
    Write $(d) \coloneqq N_{L/\Q}(\mf{d}_S)$ and define $d_S \coloneqq \abs{d}^{1/[L : \Q]}$ be the root discriminant.
\end{defi}

\begin{rema}
    In \cite[\S2]{Yua18}, they also define the discriminant of a CM-type, which we temporarily denote $d_{\Phi, YZ}$.
    We normalize our discriminant so $d_\Phi = d_{\Phi, YZ}^{1/[E_\Phi : \Q]}$.
\end{rema}

Recall that $\mu_{Art} \in \mc{CM}^0$ was the function given on Artin characters $\chi$ by $\mu_{Art}(\chi) = \log f_\chi$.
For each finite prime $p \in \Z$, we define the function $\mu_{Art, p} \in \mc{CM}^0$ as
\[\mu_{Art, p}(\chi) = v_p(f_\chi),\]
where $v_p$ is the $p$-adic valuation normalized so that $v_p(p) = 1$.
In this way, we have
\[\mu_{Art}(a) = \sum_p \mu_{Art, p}(a)\log p.\]
There is the following proposition of Colmez.

\begin{prop}[{\cite[Lem. I.2.4, Prop. I.2.6]{Col93}}]\label{prop:mu_Art local contribution}
    Let $E/\Q_p$ be a finite extension and $L\subset \clos{\Q_p}$ be a finite Galois extension of $\Q_p$ that contains all the conjugates of $E$.
    Let $f_E \coloneqq [\mc{O}_E/\mf{m}_E : \mb{F}_p]$ and let $\phi_p$ denote the Frobenius map.
    Fix an embedding $E \subset \clos{\Q_p}$ and for $\sigma \in H_E$, let $i(\sigma) \in \Z/f_E\Z$ be such that $\sigma$ induces the map $\phi_p^{i(\sigma)}$ on the residue field.
    Let $\sigma, \tau \in H_E$ and let $a_{\sigma, \tau} \colon \Gal(\clos{\Q_p}/\Q_p) \to \Q$ be the function such that $a_{\sigma, \tau}(x) = 1$ if $x\sigma = \tau$ and $0$ otherwise.
    \[\mu_{Art, p}(a_{\sigma, \tau}) = \begin{cases}
        0 & \text{if $i(\sigma) \neq i(\tau)$,}\\
        v_p(\mc{D}_{\tau(E)}) & \text{if $\sigma = \tau$,}\\
        -v_p(\tau(\pi_E) - \sigma(\pi_E)) & \text{if $i(\tau) = i(\sigma)$ and $\sigma \neq \tau$.}
    \end{cases}\]
\end{prop}

\begin{prop}\label{prop:Explicit discriminant of partial CM type}
    Let $d_E \in \Z$ denote the absolute discriminant of $E$.
    Then
    \[\mu_{Art}(A_\phib^0) = \frac{\abs{\phib}}{2[E:\Q]}\log |d_E| - \frac{1}{2}\log d_\phib + \frac{1}{4} \log d_\Sigma.\]
\end{prop}
\begin{proof}
    We can perform this computation locally for all $p$, and we first precisely determine what the local contribution of $\mf{d}_S$ looks like for $S \subset H_E$.
    Fix a prime $p$ and let $\mf{p}$ be a prime of $E$ above $p$ and let $E_0 \subset E_\mf{p}$ be the largest unramified extension of $\Q_p$ lying in $E_\mf{p}$.
    Let $L \subset \clos{\Q_p}$ be a finite Galois extension of $\Q_p$ containing all the conjugates of $E_\mf{p}$.
    The ideal $\mf{d}_{S, p}$ of $E \otimes \Z_p$ splits into a product of ideals for each embedding of $E_0$ in $L$ and so fix such an embedding and let $T$ denote the embeddings of $S$ that induce this embedding of $E_0$.
    Let $\pi_E$ be the generator for $\mc{O}_{E, \mf{p}}$ over $\mc{O}_{E_0}$ and let $p_T(t) = \prod_{\sigma \in T} (t - \sigma(\pi_E)) \in \mc{O}_L[t]$.
    The image of $\mc{O}_{E_\mf{p}} \otimes_{\mc{O}_{E_0}} \mc{O}_L$ under $f_T$ is isomorphic to $\mc{O}_L[t]/p_T[t]$ and so the local contribution of $d_{T, \mf{p}}$ is given by
    \[v_p(d_{T, \mf{p}}) = \sum_{\sigma, \tau \in T, \sigma \neq \tau} v_p(\sigma(\pi_E) - \tau(\pi_E)).\]
    
    By abuse of notation, we view $\Sigma \in H_E$ as the subset of places of $E$ whose restriction to $F$ lie in $\Sigma$.
    We write $A_\phib$ as
    \[A_\phib(x) = \abs{\phib \cap x \phib} + \frac{\abs{\phib}}{2} - \frac{1}{2} \abs{\Sigma \cap x\Sigma},\]
    which, written in terms of $a_{\sigma, \tau}$, is
    \[A_\phib = \sum_{\sigma \in \phib} \paren*{\sum_{\tau \in \phib} a_{\sigma, \tau} + \frac{1}{2} - \sum_{\tau \in \Sigma} \frac{1}{2}a_{\sigma, \tau}}.\]
    By Proposition \ref{prop:mu_Art local contribution} and the previous discussion, we have that
    \[\mu_{Art, p}(A_\phib) + \mu_{Art, p}(A_{\conjugate{\phib}}) = \sum_{\sigma \in \Sigma}\frac{1}{2}v_p(\mc{D}_{\sigma(E)}) - \frac{1}{2}v_p(d_\phib) - \frac{1}{2}v_p(d_{\conjugate{\phib}}) + \frac{1}{2}v_p(d_\Sigma).\]
    However, since $v_p(d_\phib) = v_p(d_{\conjugate{\phib}})$ and $A_\phib = A_{\conjugate{\phib}}$, we have that
    \[\mu_{Art, p}(A_\phib) = \sum_{\sigma \in \phib}\frac{1}{4}v_p(\mc{D}_{\sigma(E)}) - \frac{1}{2}v_p(d_\phib) + \frac{1}{4}v_p(d_\Sigma)\]
    and taking an average gives
    \[\mu_{Art, p}(A_\phib^0) = \frac{\abs{\phib}}{2[E:\Q]}v_p(d_E) - \frac{1}{2}v_p(d_\phib) + \frac{1}{4}v_p(d_\Sigma).\]
\end{proof}

\begin{rema}
    When $\abs{\phib}$ is a single place $\sigma$, then $d_\phib = 0$ and $d_\Sigma = d_{E/F}^{1/[F:\Q]}$, where $d_{E/F}$ is the norm of the relative discriminant of $E/F$ and the result gives
    \[\mu_{Art}(A_{\set{\sigma}}^0) = \frac{1}{2[E:\Q]}\log d_Ed_{E/F} = \frac{1}{[E:\Q]}\log d_Fd_{E/F},\]
    which recovers the log terms of \cite{And18,Yua18}.
\end{rema}

\begin{prop}\label{prop:Height of Partial CM Type L Func}
    Assume Conjecture \ref{conj:Colmez}.
    Let $\sigma, \tau \in H_E$ be two distinct places such that $\sigma \neq \conjugate{\tau}$ and let $\phib = \set{\sigma, \tau}$.
    Let $L \subset \Q^{\mathrm{cm}}$ be a finite Galois CM-extension of $\Q$ that contains all the conjugates of $E$ and let $H \coloneqq \Gal(L/\sigma(E)) \subset \Gal(L/\Q)$ be the stabilizer of $\sigma(E) \subset L$.
    Let $x_{\sigma, \tau} \in \Gal(L/\Q)$ be an element such that $x_{\sigma, \tau}\sigma = \tau$ and let $c \in \Gal(L/\Q)$ denote complex conjugation.
    Then
    \begin{align*}
        h(\phib) = &-  \frac{1}{2}\sum_{\chi}\paren*{\sum_{y \in H} \chi(x_{\sigma, \tau}y) + \chi(x_{\sigma, \tau}cy) + \frac{1}{\chi(x_{\sigma, \tau}y)} + \frac{1}{\chi(x_{\sigma, \tau}cy)}}\frac{L'(\chi, 0)}{L(\chi, 0)}\\
        &-\frac{1}{g}\frac{L'(\chi_{E/F}, 0)}{L(\chi_{E/F}, 0)} - \frac{1}{4[E:\Q]}\log d_E + \frac{1}{4}\log d_\phib - \frac{1}{8} \log d_\Sigma + \log 2\pi,
    \end{align*}
    where the summation runs over all irreducible characters $\chi$ of $\Gal(L/\Q)$.
\end{prop}
\begin{proof}
    The discriminant term is calculated in Proposition \ref{prop:Explicit discriminant of partial CM type} and so it suffices to express $A_\phib^0$ in terms of Artin characters in order to calculate what the $L$-function terms looks like.
    We have the decomposition 
    \[A_\phib^0 = 1 + b_{\sigma, \sigma}^0 + b_{\sigma, \tau}^0 + b_{\tau, \sigma}^0 + b_{\tau, \tau}^0.\]
    It is straightforward to see that
    \[b_{\sigma, \sigma} = \frac{1}{2g} \Ind_{E/F}^{E/\Q} \chi_{E/F}\]
    which gives the first term, and $Z(1) = -\log 2\pi$ which gives that term.
    Now we show that if $\chi$ is a character of $\Gal(L/\Q)$, then
    \[\inner{\chi, b_{\sigma, \tau}^0} = \sum_{h \in H} \chi(x_{\sigma, \tau}y).\]
    This follows from the fact that
    \[b_{\sigma, \tau} = \frac{1}{2}\mathbbm{1}_{x_{\sigma, \tau}H} - \frac{1}{2}\mathbbm{1}_{x_{\sigma, \tau}cH},\]
    where $\mathbbm{1}_{H}$ is the indicator function, and that after averaging we still have
    \[\inner{\chi, \mathbbm{1}_{x}^0} = \chi(x)\]
    for any $x \in \Gal(L/\Q)$.
\end{proof}

\begin{exam}\label{exam:Height of pair large}
	We note that if $\phib$ is a full CM-type of $E$, then $h(\phib)$ corresponds to the Faltings height of any CM abelian variety of type $(\mc{O}_E, \phib)$.
	These heights can drastically vary for a fixed $E$, even in the case of $\abs{\phib} = 2$, as we now demonstrate.
	
	Let $E = \Q(\sqrt{-1}, \sqrt{p}) \subset \C$ be the biquadratic field with the embeddings given by $H_E = \set{1, c, \sigma, c\sigma}$, where $c$ is complex conjugation on $\C$ and $\sigma$ fixes $i$ but sends $\sqrt{p} \to -\sqrt{p}$.
	Then there are two non-equivalent CM-types on $E$ given by $\Phi_1 = \set{1, \sigma}$ and $\Phi_2 = \set{1, \sigma c}$.
	However, these are not primitive CM-types and induced from CM-types on $E_1 = \Q(\sqrt{-1})$ and $E_2 = \Q(\sqrt{-p})$, respectively.
	For $i \in \set{1, 2}$, let $h(E_i)$ denote the Faltings height of an elliptic curve with CM by $\mc{O}_{E_i}$.
	Then, we have that $h(\Phi_1) = 2h(E_1)$ and $h(\Phi_2) = 2h(E_2)$.
	The former is constant, whereas conjecturally under GRH, the latter grows logarithmically with $p$.
    We briefly explain.
    Conjecture \ref{conj:Colmez} holds for elliptic curves (as it is just a reformulation of the Chowla--Selberg Formula) and hence we can express the Faltings height $h(E_2)$ in terms of the logarithmic derivative $\frac{L'(\chi_{E_2}, 0)}{L(\chi_{E_2}, 0)}$, where $\chi_{E_2}$ is the nontrivial quadratic character associated with $E_2$.
    Taking the logarithmic derivative of the functional equation of the completed $L$-function gives the equation
    \[h(E_2) = \frac{1}{4}\log d_{E_2} + \frac{1}{2} \frac{L'(\chi_{E_2}, 1)}{L(\chi_{E_2}, 1)} -\frac{\gamma}{2} - \log 2\pi,\]
    where $\gamma$ is Euler's constant.
    Under the assumption of GRH, the logarithmic derivative at $1$ term is $O(\log \log d_{E_2})$ (see \cite[Thm. 3]{IMS09}) and hence the dominant term is $\frac{1}{4} \log d_{E_2}$.
	Thus, for fixed degree $2g = [E:\Q]$, the ratio and difference between two heights of partial CM-types of a given field can be arbitrarily large.
\end{exam}

\section{Decomposition of Faltings Heights}\label{sec:Decomposition of Faltings Height}

\subsection{Faltings Height}
    
    We first define the Faltings height of an abelian variety defined over a number field.
    It can be defined as the degree of a Hermitian line bundle but we will give an explicit description in terms of valuations to more closely align with heights given in \cite{Col93}.
    If $K \subset \clos{\Q}$ is a number field, let $H_K \coloneqq \Hom(K, \clos{\Q})$.
    For each prime $p$, fix an embedding of $\clos{\Q}$ into $\clos{\Q_p}$ and also an embedding of $\clos{\Q}$ into $\C$.
    Let $\mc{O}_p$ denote the ring of integers in $\clos{\Q_p}$. In this way, we can identify $H_K$ with $\Hom(K, \clos{\Q}_p)$ and $\Hom(K, \C)$.
    For each prime $p$, let $v_p$ denote the unique extension of the valuation on $\Q_p$ to $\clos{\Q_p}$ with $v_p(p) = 1$.
    With our choices of embeddings, we can discuss $v_p(\sigma(\alpha))$ for any $\alpha \in K^\times$ and $\sigma \in H_K$.
    
    Let $A$ be an abelian variety of dimension $g$ defined over a number field $K \subset \C$, with semi-stable reduction (after possibly enlarging $K$ if necessary).
    Let $\mc{A}$ be the N\'eron model over $\mc{O}_K$ and let $\Omega_{\mc{A}/\mc{O}_K}$ denote the sheaf of relative differentials.
    Let $\Omega(\mc{A}) \coloneqq H^0(\mc{A}, \Omega_{\mc{A}/\mc{O}_K})$ be the $\mc{O}_K$-module and let $\omega(\mc{A}) \coloneqq H^0(\mc{A}, \Omega_{\mc{A}/\mc{O}_K}^g)$ denote the Hodge bundle of $\mc{A}$.
    This is rank $1$ projective $\mc{O}_K$-module and $\omega(A) \coloneqq \omega(\mc{A}) \otimes_{\mc{O}_K} K \cong H^0(A, \Omega_{A/K}^g)$ is a $1$-dimensional $K$-vector space.
    For each $\omega \in \omega(\mc{A})$ and $\sigma \in H_K$, we let $\omega^\sigma$ be the base change of $\omega$ viewed as an element of $\omega(\mc{A}) \otimes_{\mc{O}_K, \sigma} R \cong H^0(A^\sigma, \Omega^g_{A^\sigma/R})$, where $R = \clos{\Q_p}$ or $\C$ and $A^\sigma \coloneqq A \otimes_{K, \sigma} R$.
    We view $\omega(\mc{A})$ as a lattice in $H^0(A, \Omega_{A/K}^g)$.
    This lattice allows us to define a valuation on $H^0(A, \Omega_{A/K}^g)$ for each rational prime $p$ and embedding $\sigma \in \Hom(K, \clos{\Q_p})$ by setting $v_p(\omega^\sigma) = 0$ if $\omega^\sigma$ generates $\omega(\mc{A}) \otimes_{\mc{O}_K, \sigma} \mc{O}_p$ as an $\mc{O}_p$-module and $v_p(\alpha\omega^\sigma) = v_p(\sigma(\alpha)) + v_p(\omega^\sigma)$ for any $\alpha \in K$.
    This gives us the local contribution to the height.

    We now define a norm for each archimedean place $\sigma \colon K \to \C$.
    We can view $\omega^\sigma \in H^0(A^\sigma, \Omega_{A^\sigma/\C}^g)$ as a holomorphic $g$-form on $A^\sigma$ and we set
	\[\norm{\omega^\sigma} \coloneqq \abs*{\int_{A^\sigma(\C)}\omega^\sigma \wedge \conjugate{\omega^\sigma}}^{\frac{1}{2}}.\] 
	
    \begin{defi}\label{defn:Faltings Height}
        The \emph{Faltings height} of the abelian variety $A/K$ is the sum
        \[h(A) \coloneqq \frac{-1}{[K:\Q]} \paren*{\sum_{\sigma \in H_K} \log \norm{\omega^\sigma} - \sum_{p < \infty} \sum_{\sigma \in H_K} v_p(\omega^\sigma)\log p},\]
        for a choice of $\omega \in \omega(A) \backslash \{0\}$.
        This is well defined and independent of the choice of $\omega$ by the product formula.
    \end{defi}

    \begin{rema}
        Our choice of norm at infinity differs from that of \cite{Yua18}.
        Their metric, which we denote by $\norm{\omega^\sigma}_{YZ}$, is given by
        \[\norm{\omega^\sigma}_{YZ} = \frac{1}{(2\pi)^{g/2}} \norm{\omega^\sigma}.\]
        Thus, the Faltings height given in \cite{Yua18}, which we denote by $h_{YZ}$, differs from the one used in this article in the following way:
        \[h_{YZ}(A) = h(A) - \frac{g}{2}\log 2\pi.\]
    \end{rema}

    \subsection{Yuan--Zhang Decomposition}
	
We recall the results of \cite{Yua18} decomposing the Faltings height of a CM-type $\Phi$ into its constituent embeddings $\tau \in \Phi$.
    To decompose the height, we first decompose the Hodge bundle into its eigenspaces.
	
    Now suppose that $A$ has complex multiplication of type $(\mc{O}_E, \Phi)$.
    Let $A^t$ be the dual abelian variety of $A$.
    This is an abelian variety with CM type $(\mc{O}_E, \conjugate{\Phi})$.
    We have the canonical de Rham perfect pairing
    \[\inner{\cdot, \cdot}_{\dR} \colon H^1_{\dR}(\mc{A}/\mc{O}_K) \times H^1_{\dR}(\mc{A}^t/\mc{O}_K) \to \mc{O}_K.\]
    For $\tau \in H_E$, let $H^1(A)_\tau$ be the $\tau$-eigencomponent of $H^1_{\dR}(A(\C), \C)$ on which $E$ acts via $\tau \colon E \to \C$, and let $H^1(\mc{A})_\tau \coloneqq H^1(A)_\tau \cap H^1_{\dR}(\mc{A}/\mc{O}_K)$.
    This pairing pairs the $\tau$-eigencomponent of $A$ with the $\conjugate{\tau}$-eigencomponent of $A^t$.
    In this way, the pairing decomposes into a sum of orthogonal pairings
    \[\inner{\cdot, \cdot}_{\dR, \tau} \colon H^1(A)_\tau \times H^1(A^t)_{\conjugate{\tau}} \to K.\]

    This pairing also respects the Hodge filtration $\Fil^1H^1(A/K) = \Omega(A)$ and thus gives a pairing one-dimensional spaces
    \[\Omega(A)_\tau \times \Omega(A^t)_{\conjugate{\tau}} \to \C\]
    whenever $\tau \in \Phi$, and $\Omega(A)_{\conjugate{\tau}} = 0$.
    This gives a Hermitian norm on the line bundle
	\[N(A, \tau) \coloneqq \Omega(A)_\tau \otimes \Omega(A^t)_{\conjugate{\tau}},\]
	
	We can extend $N(A, \tau)$ to the N\'eron model of $A$.
    Suppose that $\End(A)$ is defined over $K$ and $K$ contains all embeddings of $E$ into $\clos{\Q}$.
    If $\mc{A}$ is the N\'eron model over $\mc{O}_K$ as before, define
	\[\Omega(\mc{A})_\tau \coloneqq H^0(\mc{A}, \Omega_{\mc{A}/\mc{O}_K}^1) \otimes_{\mc{O}_K \otimes \mc{O}_E, \tau} \mc{O}_K\]
	for each $\tau\colon E \to K$.
    This gives a fractional ideal inside of $\Omega(A)_\tau \cong K$ which allows us to define a valuation of elements $v_{p, YZ}$ as before.
    We add the subscript $YZ$ because we will later consider another valuation on $\Omega(A)_\tau$.
    
    For each archimedean place of $K$, we use the aforementioned Hermitian norm $\norm{\cdot}$ on the generic fiber of $\Omega(\mc{A})_\tau \otimes \Omega(\mc{A}^t)_{\conjugate{\tau}}$, and thus we get a metrized line bundle
	\[\widehat{\mc{N}(\mc{A}, \tau)} \coloneqq (\Omega(\mc{A})_\tau \otimes \Omega(\mc{A}^t)_{\conjugate{\tau}}, \norm{\cdot}).\]
	\begin{defi}
		If $A$ is an abelian variety of CM-type $(E, \Phi)$ and $\tau\colon E \to \C$, then the $\tau$-part of the Faltings height of $A$ is
		\begin{align*}
		    h(A, \tau) \coloneqq& \frac{1}{2[K:\Q]} \widehat{\deg}\widehat{\mc{N}(\mc{A}, \tau)} \\
            =& \frac{-1}{2[K : \Q]} \paren*{\sum_{\sigma \in H_K} \log \abs{\inner{\omega^\sigma_\tau, \eta^\sigma_{\clos{\tau}}}_{\dR}} - \sum_{\substack{p < \infty \\ \sigma \in H_K}} (v_p(\omega^\sigma_\tau) + v_p(\eta^\sigma_{\clos{\tau}}))\log p}
		\end{align*}
        for any nonzero choices of $\omega_\tau \in \Omega(A)_\tau$ and $\eta_{\clos{\tau}} \in \Omega(A^t)_{\clos{\tau}}$.
	\end{defi}
	Note that if $\tau \not \in \Phi$, then $\mc{N}(\mc{A}, \tau) = 0$ and so $h(A, \tau)$ is $0$ as well.
	
	Just as with the Faltings height, this $\tau$-component is independent of the abelian variety itself.
    Thus, we will write $h(\Phi, \tau)$ for $h(A, \tau)$.
	
	\begin{theo}[{{\cite[Thm 2.2]{Yua18}}}]\label{thm:fullvscomponent}
		If $A$ has CM of type $(\mc{O}_E, \Phi)$, the height $h(A, \tau)$ depends only on the pair $(\Phi, \tau)$.
	\end{theo}

\subsection{Colmez Decomposition} \label{subsec:Colmez Decomposition}

We review the decomposition of the Faltings heights of CM-abelian varieties $h(\Phi)$ given in \cite{Col93}.
Let $p$ be a prime number and let $\sigma \in H_K$ and $\tau \in H_E$.
The projection $\mc{O}_K \otimes \mc{O}_E \to \mc{O}_K$ given by $\sigma$ gives $H^1(\mc{A})_{\tau}$ the structure of an $\mc{O}_K$-module.
We then define $v_{p, C}(\omega)$ if $\omega \in H^1(A)_{\tau} \otimes_{K, \sigma} \clos{\Q_p}$ (resp. $\omega \in H^0(A, \Omega_A^g) \otimes_{K, \sigma} \clos{\Q_p}$) by $v_{p, C}(\omega) = 0$ if $\omega$ is a generator of the $\mc{O}_p$-module $H^1(\mc{A})_\tau \otimes_{\mc{O}_K, \sigma} \mc{O}_p$ (resp. $H^0(\mc{A}, \Omega_\mc{A}^g) \otimes_{\mc{O}_K, \sigma} \mc{O}_p$) and $v_{p, C}(\alpha\omega) = v_p(\alpha) + v_{p, C}(\omega)$ if $\alpha\in \clos{\Q_p}$.
We use the superscript $v_{p, C}$ to contrast it with the valuation $v_{p, YZ}$ defined in the previous subsection.
On $H^0(\mc{A}, \Omega^g_\mc{A})$, this is identical to the valuation given for the Faltings height.

Let $\omega_\tau$ be a basis element of the one-dimensional $K$-vector space $H^1(A)_\tau$.
We assume $K/\Q$ is Galois and identify $H_E$ with $\Hom(E, K)$.
In this way, we can view $\sigma \tau$ as an element of $H_E$ and so $\omega_\tau^\sigma \in H^{\sigma \tau}(A^\sigma)$.
Complex conjugation also induces a topological isomorphism between $A^\sigma(\C)$ and $A^{\clos{\sigma}}(\C)$ and thus another isomorphism denoted $c$ between $H_1(A^\sigma(\C), \Q)$ and $H_1(A^{\clos{\sigma}}(\C), \Q)$.
We have $c(\alpha u) = \alpha c(u)$ if $\alpha \in E$ and $u \in H_1(A^\sigma(\C), \Q)$.
Choose, for each $\sigma \in H_K$, a nonzero element $u_\sigma \in H_1(A^\sigma(\C), \Q)$, such that $u_{\clos{\sigma}} = c(u_\sigma)$ and define
	\[\inner{\omega_\tau^\sigma, \omega_{\clos{\tau}}^\sigma, u_\sigma} \coloneqq \paren*{\inner{\omega_\tau^\sigma, u_\sigma} \frac{2\pi i}{\inner{\omega_{\clos{\tau}}^\sigma, u_\sigma}}}^{1/2}.\]
Colmez proves the following.

\begin{lemm}[{\cite[Lem. II.2.9]{Col93}}]\label{lem:II.2.9}
    The quantity
    \[\frac{-1}{[K : \Q]}\sum_{\sigma \in H_K} \paren*{\log \abs{\inner{\omega_\tau^\sigma, \omega_{\clos{\tau}}^\sigma, u_\sigma}}_\infty - \frac{1}{2} \sum_{p < \infty} \log p (v_{p, C}(\omega_\tau^\sigma) - v_{p, C}(\omega_{\clos{\tau}}^\sigma))}\]
    only depends on $E, \tau, \Phi$ and not on the choice of $X, K, \omega_\tau, \omega_{\clos{\tau}}$, or $u_\sigma$. We denote it $ht(\Phi, \tau)$.
\end{lemm}

\section{Comparison}\label{sec:Comparison of YZ and Colmez}

We will prove that the component heights $h(\Phi, \tau)$ and $ht(\Phi, \tau)$ are the same up to a constant.

\begin{theo}\label{thm:ColmezYuanComparison}
    These two heights decompositions differ by $\log \abs{d_E}$ where $d_E$ is the absolute discriminant of $E$.
    Namely, for any CM field $E$, CM-type $\Phi \subset H_E$, and $\tau \in \Phi$, we have
    \[h(\Phi, \tau) = ht(\Phi, \tau) + \frac{1}{2} \log 2\pi + \frac{1}{4g} \log \abs{d_E} - \mu_{Art}(a_{\Phi, \tau}^0).\]
\end{theo}
\begin{proof}
    The Weil-pairing
    \[H_1(A^\sigma(\C), \Z) \times H_1(A^{t, \sigma}(\C), \Z) \to 2\pi i\Z\]
    is a perfect pairing between rank $2g$ $\Z$-modules.
    Moreover, both admit a compatible $\mc{O}_E$-action so that we may identify each with fractional ideals of $E$ and so that the pairing becomes a twist of the trace pairing $\inner{u, v} = 2\pi i\Tr_{E/\Q}(u \conjugate{v})$.
    
    For each $\sigma \in H_K$, choose nonzero elements $u_\sigma \in H_1(A^\sigma(\C), \Z)$ and $v_\sigma \in H_1(A^{t, \sigma}(\C), \Z)$ such that $u_{\clos{\sigma}} = c(u_\sigma)$ and both are equal to $1$ under our identification of $H_1(A^\sigma, \Z)$ and $H_1(A^{t, \sigma}, \Z)$ with fractional ideals of $E$.
    We can find elements $\alpha_{i, \sigma}, \beta_{i, \sigma} \in E$ that form a dual basis of $H_1(A^\sigma, \Z)$ and $H_1(A^{t, \sigma}, \Z)$ so that $\inner{\alpha_{i, \sigma}u_\sigma, \beta_{j, \sigma}v_\sigma} = 2\pi i\delta_{ij}$.
    Let $\mf{a}_\sigma$ denote the fractional ideal generated by the $\alpha_{i, \sigma}$ and $\mf{b}_\sigma$ the fractional ideal generated by the $\beta_{i, \sigma}$.
    Since the pairing is the trace pairing, we have that $\mf{b}_\sigma = \conjugate{\mf{a}_\sigma^{-1}\mf{d}_E^{-1}}$, where $\mf{d}_E$ is the different ideal.

    For each $\tau \in H_E$, choose basis elements $\omega_\tau \in H^1(A)_\tau$ and $\eta_\tau \in H^1(A^t)_\tau$ so that $\omega_{\clos{\tau}} = c\omega_\tau$ and $\eta_{\clos{\tau}} = \eta_\tau$.
    The Yuan--Zhang height is
    \begin{align*}
    	2h(\Phi, \tau) = \frac{-1}{[K : \Q]} \sum_{\sigma \in H_K}\Bigg(&\log \abs{\inner{\omega_\tau^\sigma, \eta_{\clos{\tau}}^\sigma}}\\
    	& - \sum_{p < \infty} \paren*{v_{p, YZ}(\omega_\tau^\sigma) + v_{p, YZ}(\eta_{\clos{\tau}}^\sigma)}\log p\Bigg).
    \end{align*}
    The pairing $H^1_{\dR}(A/K) \times H^1_{\dR}(A^t/K)$ is dual to the pairing on homology and so we can write
    \begin{align*}
        \inner{\omega_\tau^\sigma, \eta_{\conjugate{\tau}}^\sigma} =&\frac{1}{2\pi i}\sum_i \inner{\omega_\tau^\sigma, \alpha_i u_\sigma}\inner{\eta_{\conjugate{\tau}}^\sigma, \beta_i v_\sigma} \\
        =&\frac{1}{2\pi i}\inner{\omega_\tau^\sigma, u_\sigma}\inner{\eta_{\conjugate{\tau}}^\sigma, v_\sigma} \sum_i \sigma\tau(\alpha_i)\sigma \conjugate{\tau}(\beta_i)\\
        =&\frac{1}{2\pi i}\inner{\omega_\tau^\sigma, u_\sigma}\inner{\eta_{\conjugate{\tau}}^\sigma, v_\sigma}.
    \end{align*}
    The second equality is because the pairing respects $\mc{O}_E$-action and $\omega^\sigma_\tau$ is in the $\sigma\tau$-eigencomponent of $H^1_{\dR}(A^\sigma/K)$.
    The third equality is because the summation is equal to $1$ for the following reason.
    If we let $M$ be the matrix $\set{\tau_i \alpha_j}_{i, j}$ for all places $\tau_i$ of $E$ and $N$ be the matrix $\set{\tau_i \conjugate{\beta_j}}_{i, J}$, then $M^tN = I_{2g}$ since $\Tr_{E/\Q}(\alpha_i \conjugate{\beta_j}) = \delta_{ij}$ and hence $MN^t = I_{2g}$ as well.
    
    Combining the two, we see that
    \begin{equation}\label{eqn:Yuan-Zhang Height Comp}
        \begin{split}
            2 h(\Phi, \tau) = \frac{-1}{[K : \Q]} \sum_{\sigma \in H_K} \bigg(&\log \abs*{\inner{\omega_\tau^\sigma, u_\sigma}} + \log \abs*{\inner{\eta_{\conjugate{\tau}}^\sigma, v_\sigma}} - \log2\pi\\
            &- \sum_{p < \infty} \log p (v_{p, YZ}(\omega_\tau^\sigma) + v_{p, YZ}(\eta_{\conjugate{\tau}}^\sigma))\bigg)\\
        \end{split}
    \end{equation}
    
    The height given by Colmez is Galois invariant and thus $ht(\Phi, \tau) = ht(\conjugate{\Phi}, \conjugate{\tau})$.
    It is also invariant under the choice of abelian variety so we will choose $A$ for calculating $ht(\Phi, \tau)$ and $A^t$ for calculating $ht(\conjugate{\Phi}, \conjugate{\tau})$.
    By \cite[Lem. II.2.13]{Col93}, for each $\sigma \in H_K$ and $\tau \in H_E$, there exists a $\beta_{\sigma, \tau} \in K^\times$ (and $\gamma_{\sigma, \conjugate{\tau}} \in K^\times$) such that
    \[\inner{\omega_\tau^\sigma, u_\sigma} \inner{\omega_{\clos{\tau}}^\sigma, u_\sigma} = \sigma(\beta_{\sigma, \tau}) 2\pi i.\]
    Putting this in gives that $ht(\Phi, \tau) + ht(\conjugate{\Phi}, \conjugate{\tau})$ is equal to
    \begin{align*}
    	\frac{-1}{[K : \Q]} \sum_{\sigma \in H_K} &\bigg(\log \paren*{\abs*{\inner{\omega_\tau^\sigma, u_\sigma}}\abs*{\inner{\eta_{\conjugate{\tau}}^\sigma, v_\sigma}}} - \frac{1}{2}\log\paren*{ \abs{\sigma(\beta_{\sigma, \tau})}\abs{\sigma(\gamma_{\sigma, \conjugate{\tau}})}}\\
        &- \frac{1}{2}\sum_{p < \infty} \log p (v_{p, C}(\omega_\tau^\sigma) - v_{p, C}(\omega_{\conjugate{\tau}}^\sigma) + v_{p, C}(\eta_{\conjugate{\tau}}^\sigma) - v_{p, C}(\eta_{\tau}^\sigma))\bigg)\\
    \end{align*}
    The element $\sigma(\beta_{\sigma, \tau}) \in K^\times$ satisfies the product formula.
    Moreover at local places $p$, the valuation is given by
    \[v_p(\sigma(\beta_{\sigma, \tau})) = v_{p, C}(\omega_\tau^\sigma) + v_{p, C}(\omega_{\conjugate{\tau}}^\sigma) - v_p(\sigma \tau(\mf{a}_\sigma)) - v_p(\sigma \conjugate{\tau}(\mf{a}_\sigma)).\]
    The same, \textit{mutatis mutandis}, holds true for $\sigma(\gamma_{\sigma, \conjugate{\tau}})$. 
    Using that $\mf{a}_\sigma\conjugate{\mf{b}_\sigma} = \conjugate{\mf{d}_E^{-1}}$, we combine everything to get
    \begin{equation}\label{eqn:Colmez Height Comp}
        \begin{split}
            ht(\Phi, \tau) + ht(\conjugate{\Phi}, \conjugate{\tau}) = \frac{-1}{[K : \Q]} \sum_{\sigma \in H_K} \Bigg(&\log \abs*{\inner{\omega_\tau^\sigma, u_\sigma}} + \log \abs*{\inner{\eta_{\conjugate{\tau}}^\sigma, v_\sigma}}\\
            &- \sum_{p < \infty} \log p \bigg(v_{p, C}(\omega_\tau^\sigma) + v_{p, C}(\eta_{\conjugate{\tau}}^\sigma)\\
            	& - \frac{v_p(\sigma \tau \mf{d}_E^{-1}) + v_p(\sigma \conjugate{\tau}\mf{d}_E^{-1})}{2}\bigg)\Bigg)\\
        \end{split}
    \end{equation}

    Comparing Equations \ref{eqn:Yuan-Zhang Height Comp} and \ref{eqn:Colmez Height Comp} we see that
    \begin{align*}
        h(\Phi, \tau) - ht(\Phi, \tau) - \frac{1}{2} \log 2\pi =&\\
        \frac{1}{2[K : \Q]}\sum_{\substack{\sigma \in H_K\\p < \infty}} \log p \bigg(&v_{p, YZ}(\omega_\tau^\sigma) - v_{p, C}(\omega_\tau^\sigma) + v_{p, YZ}(\eta_{\conjugate{\tau}}^\sigma) \\
        &- v_{p, C}(\eta_{\conjugate{\tau}}^\sigma) + \frac{v_p(\sigma \tau \mf{d}_E^{-1}) + v_p(\sigma \conjugate{\tau}\mf{d}_E^{-1})}{2}\bigg)
    \end{align*}

    We now compute the local terms.
    Fix $p < \infty$ and $\sigma \in H_K$, then we view $\sigma \colon K \to \clos{\Q_p}$ which gives a place $v$ of $K$.
    By \cite[Lem. II.1.2]{Col93}, the de Rham cohomology $H^1_{\dR}(\mc{A}/\mc{O}_{K_v})$ is a free $\mc{O}_{K_v} \otimes \mc{O}_E$-module of rank $1$.
    After localizing, we have 
    \[\mc{O}_{K_v} \otimes \mc{O}_E \cong \prod_{\mf{p} \mid p} \mc{O}_{K_v} \otimes \mc{O}_{E_\mf{p}}.\]
    Let $\mf{p}$ be the prime given by $\tau \colon E \to K_v$ and let $\Phi_\mf{p}$ and $\conjugate{\Phi}_\mf{p}$ denote the places of $\Phi$ contained in $H_{E_\mf{p}} \subset H_E$.
    Let $R_{\Phi_\mf{p}}$ be the image of $\mc{O}_{K_v} \otimes \mc{O}_{E_\mf{p}}$ under the map
    \[\mc{O}_{K_v} \otimes \mc{O}_{E_\mf{p}} \to \prod_{\rho \colon E_\mf{p} \to \mc{O}_{K_v}} \mc{O}_{K_v, \rho} \to \prod_{\rho \in \Phi_\mf{p}} \mc{O}_{K_v, \rho}.\]
    The valuation $v_{p, YZ}$ is given on $\mc{O}_{K_v, \tau}$ under the projection map $R_{\Phi_{\mf{p}}} \to \mc{O}_{K_v, \tau}$, which is surjective.
    On the other hand, the valuation $v_{p, C}$ is given by the intersection
    \[R_{\Phi_{\mf{p}}} \cap \paren*{\paren*{\prod_{\rho \in \Phi_\mf{p} \bs \tau} \set{0}} \times \mc{O}_{K_v, \tau}} \to \mc{O}_{K_v, \tau}.\]
    Let $\pi \in \mc{O}_{K_v, \tau}$ be such that $v_{p, C}(\pi) = 0$.
    Write $\mc{O}_{E_\mf{p}} = \Z_p[\alpha]$ so that $R_{\Phi_\mf{p}}$ is generated as an $\mc{O}_{K_v}$-module by $\set{(\rho(\alpha^i))_{\rho \in \Phi_\mf{p}}}_{0 \le i < \abs{\Phi_\mf{p}}}$.
    Then $\pi$ is such that the $\tau$-th row of the inverse of the Vandermonde matrix associated to $\alpha$ and $\Phi_\mf{p}$ has minimum valuation $-v_p(\pi)$.
    The $\tau$-th row of the inverse matrix consists of terms involving combinations of symmetric polynomials of the $\rho(\alpha)$ all divided by the product
    \[\prod_{\rho \in \Phi_\mf{p} \bs \tau} (\rho(\alpha) - \tau(\alpha)).\]
    Thus, we have that
    \[v_{p, YZ}(\omega_\tau^\sigma) - v_{p, C}(\omega_\tau^\sigma) = v_p\paren*{\prod_{\rho \in \Phi_\mf{p}\bs \tau} (\rho(\alpha) - \tau(\alpha))}.\]
    Comparing the right hand side with \cite[Lem. I.2.4, Prop. I.2.6]{Col93}, we have that
    \[v_{p, YZ}(\omega_\tau^\sigma) - v_{p, C}(\omega_\tau^\sigma) = v_p(\sigma\tau \mf{d}_E) - \mu_{Art, p}(a_{\Phi, \tau})/ \log p.\]
    Thus, we end up with
    \begin{align*}
    	h(\Phi, \tau) - ht(\Phi, \tau) = &\frac{1}{2[K : \Q]}\sum_{\substack{\sigma \in H_K\\p < \infty}} \bigg(\log p \paren*{\frac{v_p(\sigma \tau \mf{d}_E) + v_p(\sigma \conjugate{\tau}\mf{d}_E)}{2}}\\
    	& - \mu_{Art, p}(a_{\Phi, \tau}) - \mu_{Art, p}(a_{\conjugate{\Phi}, \conjugate{\tau}})\bigg) + \frac{1}{2} \log 2\pi \\
    	=&\frac{1}{2[E : \Q]}\log \abs{d_E} - \mu_{Art, p}(a_{\Phi, \tau}^0) + \frac{1}{2} \log 2\pi ,
    \end{align*}
    thus proving the result.
\end{proof}

As a direct consequence of this Theorem and our comparison Theorem \ref{thm:ColmezYuanComparison}, we recover the observation made in \cite{Yua18} that heights of nearby CM-types are independent of the CM-type and place.

\begin{coro}\label{cor:Nearby Height Ind of Pair}
    Let $\Phi_1$ and $\Phi_2$ be two CM-types of $E$ such that $\abs{\Phi_1 \cap \Phi_2} = g - 1$.
    Let $\tau_i = \Phi_i \bs (\Phi_1 \cap \Phi_2)$ for $i = 1, 2$ be the place at which they differ.
    Then
    \[h(\Phi_1, \tau_1) + h(\Phi_2, \tau_2)\]
    depends only on $E$.
\end{coro}
\begin{proof}
    By Theorems \ref{thm:ColmezYuanComparison} and \ref{thm:ColmezHeightFunction}, it suffices to show that $a_{\Phi_1, \tau_1}^0 + a_{\Phi_2, \tau_2}^0$ depends only on $E$.
    But, this is because a simple computation shows that for $x \in \Gal(\clos{\Q}/\Q)$,
    \[a_{\Phi_1, \tau_1}^0(x) + a_{\Phi_2, \tau_2}^0(x) = \begin{cases}
        2 & \text{if $x|_E = \mathrm{Id}$,}\\
        0 & \text{if $x|_E = c$,}\\
        1 & \text{otherwise.}
    \end{cases}\]
\end{proof}

Moreover, this leads to another proof of \cite[Cor. 2.6]{Yua18} relating the height of nearby CM-types and Faltings heights of CM abelian varieties.

\begin{coro}
    Let $\Phi_1, \Phi_2, \tau_1, \tau_2$ be as in the statement of Corollary \ref{cor:Nearby Height Ind of Pair}.
    Then
    \[\frac{1}{g2^g}\sum_\Phi h(\Phi) = \frac{h(\Phi_1, \tau_1) + h(\Phi_2, \tau_2)}{2} - \frac{1}{2} \log 2\pi - \frac{1}{4g} \log \abs{d_F},\]
    where the sum is taken over all $2^g$ CM-types of $E$.
\end{coro}
\begin{rema}
    The extra $\log 2\pi$ term not seen in \cite[Cor. 2.6]{Yua18} is because we use a different normalizing factor for the archimedean contribution in the Faltings height.
    This is expanded on in the remark after Definition \ref{defn:Faltings Height}.
\end{rema}
\begin{proof}
    We write the left hand side in terms of class functions using linearity.
    A straightforward calculation shows that for $x \in \Gal(\clos{\Q}/\Q)$,
    \[\frac{1}{g2^g} \sum_{\Phi}A_\Phi(x) = \frac{a_{\Phi_1, \tau_1}^0(x) + a_{\Phi_2, \tau_2}^0(x)}{2} = \begin{cases}
        1 & \text{if $x|_E = \mathrm{Id}$,}\\
        0 & \text{if $x|_E = c$,}\\
        \frac{1}{2} & \text{otherwise.}
    \end{cases}\]
    Thus Theorem \ref{thm:ColmezHeightFunction} gives that
    \[\frac{1}{g2^g}\sum_\Phi h(\Phi) = \frac{ht(a_{\Phi_1, \tau_1}^0 + a_{\Phi_2, \tau_2}^0)}{2} - \frac{\mu_{Art}(a_{\Phi_1, \tau_1}^0 + a_{\Phi_2, \tau_2}^0)}{4}.\]
    We have the equality $a_{\Phi_1, \tau_1}^0 + a_{\Phi_2, \tau_2}^0 = 1 + \frac{1}{g} \Ind_{E/F}^{E/\Q} \chi_{E/F}$, where $\chi_{E/F}$ is the nontrivial character of $\Gal(E/F)$.
    Thus its log conductor is
    \[\mu_{Art}(a_{\Phi_1, \tau_1}^0 + a_{\Phi_2, \tau_2}^0) = \frac{1}{g}\mu_{Art}\paren*{\Ind_{E/F}^{E/\Q} \chi_{E/F}} = \frac{1}{g} \log \abs*{\frac{d_E}{d_F}}.\]
    Now Theorem \ref{thm:ColmezYuanComparison} gives
    \begin{align*}
        \frac{1}{g2^g}\sum_\Phi h(\Phi) - \frac{h(\Phi_1, \tau_1) + h(\Phi_2, \tau_2)}{2} =& - \frac{1}{2} \log 2\pi - \frac{1}{4g} \log \abs{d_E} \\
        &+ \frac{\mu_{Art}(a_{\Phi_1, \tau_1}^0 + a_{\Phi_2, \tau_2}^0)}{4}\\
        =&- \frac{1}{2} \log 2\pi - \frac{1}{4g} \log \abs{d_F}.
    \end{align*}
\end{proof}

\bibliographystyle{alpha}
\bibliography{references}
\end{document}